\documentclass[12pt,reqno,draft]{article} 
\usepackage{amsmath,amssymb,amsthm,amsfonts, indentfirst}
\usepackage{enumerate,color,bm}
\usepackage{mathtools}
\makeatletter
\def\@cite#1#2{[{{\bfseries #1}\if@tempswa , #2\fi}]}
\renewcommand{\section}{%
\@startsection{section}{1}{\z@}
{0.5truecm plus -1ex minus -.2ex}%
{1.0ex plus .2ex}{\bfseries\large}}
\def\@seccntformat#1{\csname the#1\endcsname.\ }
\makeatother

\numberwithin{equation}{section} 
\theoremstyle{theorem}
\newtheorem{thm}{Theorem}[section]

\newtheorem{lem}[thm]{Lemma}

\theoremstyle{definition}

\newtheorem*{prth3.1}{Proof of Theorem~\ref{thm_bdd1}}
\newtheorem*{prth3.3}{Proof of Theorem~\ref{thm_bdd2}}
\newtheorem*{prth4.1}{Proof of Theorem~\ref{thm_bu1}}
\newtheorem*{prth4.2}{Proof of Theorem~\ref{thm_bu2}}

\newcommand{\ep}{\varepsilon}
\newcommand{\pa}{\partial}

\newcommand{\R}{\mathbb{R}}
\newcommand{\N}{\mathbb{N}}

\newcommand{\cl}[1]{{\overline#1}}

\newcommand{\tmax}{T_{\rm max}}

\begin{document}
\footnote[0]{
    2020{\it Mathematics Subject Classification}\/. 
    Primary: 35A01; 
    Secondary: 35B44, 35K59, 92C17.
    }
\begin{center} 
    \Large{{\bf 
    Boundedness and finite-time blow-up in a quasilinear
    parabolic--elliptic--elliptic attraction-repulsion 
    chemotaxis system
    }}%
\end{center}
\vspace{5pt}
\begin{center}
    Yutaro Chiyo, 
    Tomomi Yokota%
      \footnote{Corresponding author.}%
      \footnote{Partially supported by Grant-in-Aid for
      Scientific Research (C), No.\ 21K03278.}
    \footnote[0]{
    E-mail: 
    {\tt ycnewssz@gmail.com}, 
    {\tt yokota@rs.tus.ac.jp}
    }\\
    \vspace{12pt}
    Department of Mathematics, 
    Tokyo University of Science\\
    1-3, Kagurazaka, Shinjuku-ku, 
    Tokyo 162-8601, Japan\\
    \vspace{2pt}
\end{center}
\begin{center}    
    \small \today
\end{center}

\vspace{2pt}
\newenvironment{summary}
{\vspace{.5\baselineskip}\begin{list}{}{%
     \setlength{\baselineskip}{0.85\baselineskip}
     \setlength{\topsep}{0pt}
     \setlength{\leftmargin}{12mm}
     \setlength{\rightmargin}{12mm}
     \setlength{\listparindent}{0mm}
     \setlength{\itemindent}{\listparindent}
     \setlength{\parsep}{0pt}
     \item\relax}}{\end{list}\vspace{.5\baselineskip}}
\begin{summary}
{\footnotesize {\bf Abstract.} 
This paper deals with 
the quasilinear attraction-repulsion 
chemotaxis system
%
\begin{align*}
\begin{cases}
        u_t=\nabla\cdot \big((u+1)^{m-1}\nabla u
            -\chi u(u+1)^{p-2}\nabla v
            +\xi u(u+1)^{q-2}\nabla w\big) 
            +f(u),
          \\[1.05mm]
        0=\Delta v+\alpha u-\beta v,
          \\[1.05mm]
        0=\Delta w+\gamma u-\delta w
\end{cases}
\end{align*}
%
in a bounded domain $\Omega \subset \mathbb{R}^n$ ($n \in \mathbb{N}$) 
with smooth boundary $\partial\Omega$, 
where $m, p, q \in \mathbb{R}$, $\chi, \xi, \alpha, \beta, \gamma, \delta>0$ 
are constants. 
Moreover, it is supposed that the function $f$ satisfies $f(u)\equiv0$ 
in the study of boundedness, 
whereas, when considering blow-up, 
it is assumed that $m>0$ and $f$ is a function of logistic type 
such as $f(u)=\lambda u-\mu u^{\kappa}$ 
with $\lambda \ge 0$, $\mu>0$ and $\kappa>1$ sufficiently close to~$1$, 
in the radially symmetric setting.
In the case that $\xi=0$ and $f(u) \equiv 0$, 
global existence and boundedness 
have been proved under the condition $p<m+\frac2n$.  
Also, in the case that $m=1$, $p=q=2$ and 
$f$ is a function of logistic type, 
finite-time blow-up has been established 
by assuming $\chi\alpha-\xi\gamma>0$. 
This paper classifies boundedness and blow-up 
into the cases $p<q$ and $p>q$ 
without any condition for the sign of $\chi\alpha-\xi\gamma$
and the case $p=q$ with $\chi\alpha-\xi\gamma<0$ or $\chi\alpha-\xi\gamma>0$.
}
\end{summary}
\vspace{10pt}

\newpage

\section{Introduction} \label{Sec1}

\noindent
{\bf Background.} 
Chemotaxis is the property of cells to move 
in a directional manner in response to concentration gradients of 
chemical substances. 
The system of partial differential equations 
describing such the motion of cells 
was introduced by Keller--Segel~\cite{KS-1970}, and 
is called the chemotaxis system.
The system
\begin{align}\label{KS0}
    \begin{cases}
        u_t=\nabla \cdot \big( \nabla u
            -\chi u\nabla v\big),
          \\[1.05mm]
        v_t=\Delta v+\alpha u-\beta v
          \\[1.05mm]
    \end{cases}
\end{align}
is one of many types of the chemotaxis systems and expresses phenomena 
caused by the movement of cells as a response 
to an attractive chemical substance. 
Here the functions $u$ and $v$ idealize the cell density and the concentration of 
the chemoattractant, respectively. 
After the work~\cite{KS-1970}, there have been many extensive studies 
on the chemotaxis systems (see e.g., Osaki--Yagi~\cite{OY-2001}, 
Bellomo et al.\ \cite{BBTW-2015}, Arumugam--Tyagi~\cite{AT}). 
From the point of view of modeling, it is significant to analyze quasilinear systems 
such as the system
\begin{align*}
    \begin{cases}
        u_t=\nabla\cdot \big((u+1)^{m-1}\nabla u
            -\chi u(u+1)^{p-2}\nabla v\big), 
          \\[1.05mm]
        v_t=\Delta v+\alpha u-\beta v, 
          \\[1.05mm]
    \end{cases}
\end{align*}
where $m, p\in \R$. 
This system has been proposed by Painter--Hillen~\cite{PH-2002} 
and has been dealt with by some works
(see e.g., Cie\'slak~\cite{C-2007}, Tao--Winkler~\cite{TW-2012}; 
cf.\ also \cite{IY-2012} 
for the degenerate version of the system). 
In the other direction, in order to describe the quorum sensing effect 
that cells keep away from a repulsive chemical substance, Painter--Hillen~\cite{PH-2002} 
suggested the following attraction-repulsion chemotaxis system which 
was also introduced by Luca et al.\ \cite{LCEM-2003} 
to describe the aggregation of microglial cells in Alzheimer's disease:
\begin{align}\label{AR}
    \begin{cases}
        u_t=\nabla \cdot \big( \nabla u
            -\chi u\nabla v
            +\xi u\nabla w\big),
          \\[1.05mm]
        v_t=\Delta v+\alpha u-\beta v,
          \\[1.05mm]
        w_t=\Delta w+\gamma u-\delta w.
    \end{cases}
\end{align} 
The functions $u$, $v$ and $w$ in \eqref{AR} 
represent the cell density, the concentration of 
the chemoattractant and chemorepellent, respectively. 
The system \eqref{AR} has also been actively studied as detailed in later. 
Here we emphasize that it is meaningful to consider the system \eqref{AR} 
with diffusion, attraction and repulsion terms involving nonlinearities, that is, 
\begin{align}\label{ARquasi}
    \begin{cases}
        u_t=\nabla\cdot \big((u+1)^{m-1}\nabla u
            -\chi u(u+1)^{p-2}\nabla v
            +\xi u(u+1)^{q-2}\nabla w\big), 
          \\[1.05mm]
        v_t=\Delta v+\alpha u-\beta v,
          \\[1.05mm]
        w_t=\Delta w+\gamma u-\delta w.
    \end{cases}
\end{align}
In this paper, previous to a mathematical analysis of \eqref{ARquasi}, 
we will reduce the system to the parabolic--elliptic--elliptic 
version. 
The reduction seems to be reasonable 
because the diffusion of chemical substances 
are faster than that of cells. 
Thus we can approximate the system \eqref{ARquasi} by its parabolic--elliptic--elliptic 
version.%
\vspace{2.5mm}

\newpage
\noindent
{\bf Problem.} 
In this paper, as mentioned above, we consider the quasilinear 
parabolic--elliptic--elliptic attraction-repulsion 
chemotaxis system
%
\vspace{-0.5mm}
\begin{align}\label{prob}
    \begin{cases}
        u_t=\nabla\cdot \big((u+1)^{m-1}\nabla u
            -\chi u(u+1)^{p-2}\nabla v
            +\xi u(u+1)^{q-2}\nabla w\big) 
            +f(u),
          \\[1.05mm]
        0=\Delta v+\alpha u-\beta v,
          \\[1.05mm]
        0=\Delta w+\gamma u-\delta w,
          \\[1.05mm]
         \nabla u \cdot \nu|_{\pa \Omega}
        =\nabla v \cdot \nu|_{\pa \Omega}
        =\nabla w \cdot \nu|_{\pa \Omega}=0,
          \\[1.05mm]
        u(\cdot, 0)=u_0
    \end{cases}
\end{align}
        \vspace{-2mm}
%

\noindent
in a bounded domain $\Omega \subset \R^n$ ($n \in \N$) 
with smooth boundary $\pa\Omega$, 
where $m, p, q \in \R$, $\chi, \xi, \alpha, \beta, \gamma, \delta>0$ 
are constants,  
$\nu$ is the outward normal vector to $\pa\Omega$,  
%
\vspace{-1mm}
\begin{align}\label{u0}
    u_0 \in C^0(\cl{\Omega}),\quad 
    u_0 \ge 0\ {\rm in}\ \cl{\Omega} \quad {\rm and } \quad 
    u_0 \neq 0.
\end{align}
\vspace{-6mm}
%

\noindent
Moreover, we assume that  
\vspace{-1mm}
\begin{itemize}
\item $m \in \R$, $f(u) \equiv 0$ in the consideration of boundedness; 
\vspace{-1mm}
\item $m>0$, $f(u)=\lambda(|x|)u-\mu(|x|)u^\kappa$ ($\kappa \ge 1$) in the study of blow-up, provided that 
\begin{align}
&\Omega=B_R(0) \subset \R^n\ (n \in \N,\ n \ge 3)\ {\rm with}\ R>0, \label{omega}\\[1.5mm]
&\lambda(\cdot),\ \mu(\cdot)\ 
{\rm are\ nonnegative\ and\ continuous\ functions\ on}\ [0, R],\label{lammu}\\[1.5mm]
&\mu(r) \le \mu_1r^a\ {\rm for\ all}\ r \in [0, R]\ {\rm with\ some}\ \mu_1>0\ {\rm and}\ a \ge 0. \label{mupro}
\end{align}
\end{itemize}

\noindent
{\bf Attraction vs.\ repulsion.} 
As to the system \eqref{prob} with $p=q=2$, 
it is known that boundedness and blow-up are classified 
by the sign of $\chi\alpha-\xi\gamma$ (see e.g., Tao--Wang~\cite{TW-2013}). 
Here boundedness 
(including global existence), 
which expresses that 
$\|u(\cdot, t)\|_{L^\infty(\Omega)} \le C$ for all $t>0$ with some $C>0$, 
is interpreted as the diffusion of cells, 
and that finite-time blow-up (blow-up for short), which means that 
$\lim_{t \nearrow T}\|u(\cdot, t)\|_{L^\infty(\Omega)}=\infty$ with some $T \in (0, \infty)$,  
implies the concentration of cells. 
On the other hand, to the best of our knowledge, 
when $p \neq 2$ or $q \neq 2$, 
no results are available for boundedness and blow-up in \eqref{prob}. 
Here the powers $p, q$ determine the strengths of the effects of attraction, 
which promotes blow-up, and repulsion, which induces boundedness. 
Thus we can naturally guess as follows.
\vspace{-1mm}
\begin{align*}
\text{Boundedness and blow-up  
can be classified by the size of the powers $p, q$.}
\end{align*}
\vspace{-6mm}

\noindent
In the following we discuss this expectation. 
As will be explained later, in the case $\xi=0$ in \eqref{prob}  
it is known that boundedness holds in the case
\vspace{-1mm}
\begin{align}\label{condipm}
p<m+\frac2n, 
\end{align} 
\vspace{-5mm}

\noindent
and blow-up occurs in the opposite case. 
In view of the first equation in \eqref{prob}, 
the condition \eqref{condipm} implies that 
the effect of diffusion ``plus $\frac{2}{n}$''
is stronger than the one of attraction. 
\linebreak[4]
In the case $\xi \neq 0$ 
the system \eqref{prob} involves the repulsion term 
which is expected to work in contrast to the attraction term. 
Therefore the question arises whether the repulsion term is useful 
for deriving boundedness, that is, 
\vspace{2mm}
\makeatletter\tagsleft@true\makeatother
\begin{align}\label{Q1}\tag*{({\bf Q1})}
&\qquad\quad\text{{\it when $p<q$, does boundedness in \eqref{prob} hold 
without assuming \eqref{condipm}?}} \notag
\end{align}
\makeatletter\tagsleft@false\makeatother
\vspace{-4mm}

\noindent
In the opposite case $p>q$, we believe that blow-up can be shown 
since the effect of attraction is more dominant than that of repulsion, 
and we raise the following question.%
\vspace{2mm}
\makeatletter\tagsleft@true\makeatother
\begin{align}\label{Q2}\tag*{({\bf Q2})}
&\text{{\it When $p>q$, does blow-up in \eqref{prob} occur?}} \notag
\end{align}
\makeatletter\tagsleft@false\makeatother
\vspace{-4mm}

\noindent
Furthermore, in the case $p=q$, where the effects of 
attraction and repulsion are balanced, 
the following question arises.
\vspace{2mm}
\makeatletter\tagsleft@true\makeatother
\begin{align}\label{Q3}\tag*{({\bf Q3})}
&\text{{\it When $p=q$, are boundedness and blow-up in 
\eqref{prob}}}\\
&\text{{\it classified by the condition for the coefficients in the equations?}} \notag
\end{align}
\makeatletter\tagsleft@false\makeatother
\vspace{-4mm}

\noindent
{\bf Overview of related works.} 
Before giving answers to the above three questions,
we summarize the previous studies related to each case.
\vspace{2.5mm}

We first focus on the reduced system without repulsion term, 
\begin{align}\label{KS}
    \begin{cases}
        u_t=\nabla\cdot \big((u+1)^{m-1}\nabla u
            -\chi u(u+1)^{p-2}\nabla v\big) 
            +f(u),
          \\[1.05mm]
        \tau v_t=\Delta v+\alpha u-\beta v,
    \end{cases}
\end{align}
where $m, p \in \R$, $\chi, \alpha, \beta>0$, $\tau \in \{0,1\}$ are constants and $f$ is a function. 
In the case $\tau=1$, boundedness were shown in 
\cite{ISY-2014, TW-2012, WL-2017, Z-2015}. 
More precisely, 
Tao--Winkler~\cite{TW-2012} derived boundedness when 
$\Omega \subset \R^n$ ($n \in \N$) is a convex domain, $f(u) \equiv0$ and $p<m+\frac2n$ holds; 
after that, the convexity of $\Omega$ was removed by~\cite{ISY-2014}. 
Conversely, finite-time blow-up was obtained under the condition $p>m+\frac2n$ 
(see e.g., Winkler~\cite{W-2013}, Cie\'slak--Stinner~\cite{CS-2012, CS-2014}). 
Besides, in the critical case $p=m+\frac2n$, boundedness and blow-up 
were classified by the condition for initial data 
(\cite{BL-2013, IY-2012-2, LM-2017, M-2017}). 
Also, in the case $f(u) \le \lambda-\mu u^\kappa$ ($\lambda \ge 0$, $\mu>0$, $\kappa>1$), 
global existence of classical solutions was established by Zheng~\cite{Z-2015} under 
the condition that 
$p<\min\{\kappa-1, m+\frac2n\}$, or $p=\kappa$ if $\mu>0$ is sufficiently large. 
On the other hand, in the case $\tau=0$, 
boundedness were proved in 
\cite{LX-2015, SK-2006, WMZ-2014, Z-2015-2}. 
Particularly, in the case $\Omega=\R^n$ ($n \in \N$), Sugiyama--Kunii~\cite{SK-2006} demonstrated boundedness of 
weak solutions in the system \eqref{KS} of a degenerate type. 
Namely, in the literature the authors dealt with the case that 
$f(u) \equiv 0$, $m \ge 1$, $p \ge 2$ and $p<\min\{m+1, m+\frac2n\}$. 
Also, in the case that $p=2$ and $f(u) \le \lambda-\mu u^\kappa$ ($\lambda \ge 0$, $\mu>0$, $\kappa>1$),
boundedness were verified by~Wang et al.\ \cite{WMZ-2014} under the condition that 
$m>2-\frac2n$ if $\kappa \in (1,2)$, or $\mu>\mu^*$ if $\kappa \ge 2$ 
with some $\mu^*>0$. 
In contrast, when $m=1$, $p=2$ and $f(u)=\lambda u-\mu u^{\kappa}$ 
($\lambda \in \R$, $\mu>0$, $\kappa>1$), 
Winkler~\cite{W-2018} established finite-time blow-up; 
after that, the result was extended to the cases $p \in (1,2)$, $p=2$ and $p>1$ 
in~\cite{TY-2020}, \cite{BFL-2021} and \cite{T-2021}, respectively. 
Moreover, some related works for the system \eqref{KS} 
with nonlinear sensitivity can be found in~\cite{D-2018, F-2015, FNY-2015, JY-2019}. 
For instance, when $\tau=1$, $m=1$, $p=2$ and $f(u) \equiv 0$, 
Fujie~\cite{F-2015} showed boundedness in 
\eqref{KS} with sensitivity function $\frac{\chi}{v}$ under the condition 
$0<\chi<\sqrt{\frac{2}{n}}$. 
\vspace{2.5mm}

We next shift our focus to the attraction-repulsion system
\begin{align}\label{ARpee}
    \begin{cases}
        u_t=\nabla \cdot \big(\nabla u
            -\chi u\nabla v
            +\xi u\nabla w\big) 
            +f(u),
          \\[1.05mm]
        0=\Delta v+\alpha u-\beta v,
          \\[1.05mm]
        0=\Delta w+\gamma u-\delta w,
    \end{cases}
\end{align}
where $\chi, \xi, \alpha, \beta, \gamma, \delta>0$. 
In the case $f(u)=\lambda u-\mu u^{\kappa}$ ($\lambda\in\R$, 
$\mu>0$, $\kappa>1$), finite-time blow-up was recently proved 
in~\cite{CMTY} via the method in~\cite{W-2018} when 
$\kappa$ is sufficiently closed to $1$ and $\chi\alpha-\xi\gamma>0$ holds. 
Moreover, some related works deriving boundedness can be found in~\cite{FS-2019, NSY-2021-2, NSY-2021, NY-2020, NY-2020-2}; 
showing finite-time blow-up can be cited in~\cite{L-2021}; 
dealing with nonlinear diffusion and sensitivities can be referred 
in~\cite{CMY-2020, LMLW-2017, LMG-2016}. 
Particularly, in the two-dimensional setting, 
Fujie--Suzuki~\cite{FS-2019} established 
boundedness in 
the fully parabolic version of \eqref{ARpee} 
under the condition that $\beta=\delta$, $\chi\alpha-\xi\gamma>0$ and 
$\|u_0\|_{L^1(\Omega)}<\frac{4\pi}{\chi\alpha-\xi\gamma}$; 
note that the authors relaxed the condition for $u_0$ 
in the radially symmetric setting and removed the condition $\beta=\delta$. 
Also, Nagai--Yamada~\cite{NY-2020-2} obtained 
global existence of solutions under the condition that $\alpha=\gamma=1$, 
$\chi-\xi>0$ and $\|u_0\|_{L^1(\Omega)}=\frac{8\pi}{\chi-\xi}$ 
in the two-dimensional setting; after that, 
the authors demonstrated boundedness of solutions in \cite{NY-2020}. 
On the other hand, in the three-dimensional and radially symmetric settings, 
existence of solutions blowing up in finite time 
to the fully parabolic version of \eqref{ARpee} 
was shown by Lankeit~\cite{L-2021} 
under the conditions that $\chi\alpha-\xi\gamma>0$ and that 
$\|u_0\|_{L^1(\Omega)}=M$ with some $M>0$.%
\vspace{2.5mm}

In summary, the results on boundedness and blow-up in 
the system \eqref{prob} were obtained as follows: 
Boundedness was derived in the case $\xi=0$ 
under the condition $p<m+\frac2n$; 
blow-up was proved under the condition $\chi\alpha-\xi\gamma>0$. 
However, in previous studies, the effect of repulsion  
has not been effectively utilized. 
The purpose of this paper is to establish boundedness and blow-up 
with help of the repulsion term without the above conditions.
\vspace{2.5mm}

\noindent
{\bf Main results.} 
Before introducing our results, we mention the expected answers to 
the questions \ref{Q1}--\ref{Q3}. 
As to the questions \ref{Q1} and \ref{Q2}, 
we can give affirmative answers. 
Also, regard to the question \ref{Q3}, 
we can classify boundedness and blow-up according to 
the sign of $\chi\alpha-\xi\gamma$.
In the following we briefly state the main results which give 
the answers to the questions. 
The precise statements and their proofs will be given in Sections~\ref{Sec3},~\ref{Sec4}. 

\begin{itemize}
\item[({\bf I})] If $p<q$, then, for all initial data, 
                     the system \eqref{prob} possesses 
                     a global bounded classical solution 
                     which is unique (Theorem~\ref{thm_bdd1}).
\item[({\bf I\!I})] If $p=q$ and $\chi\alpha-\xi\gamma<0$, 
                        then, for all initial data, 
                        the system \eqref{prob} admits 
                        a unique global bounded classical solution 
                        (Theorem~\ref{thm_bdd2}).
\item[({\bf I\!I\!I})] If $p>q$, then there exist initial data such that 
                           the corresponding solutions blow up 
                           in finite time in the radial framework 
                           (Theorem~\ref{thm_bu1}).
\item[({\bf I\!V})] If $p=q$ and $\chi\alpha-\xi\gamma>0$,  
                         then there exist initial data such that 
                         the system \eqref{prob} possesses 
                         solutions blow up in finite time in the radial framework 
                         (Theorem~\ref{thm_bu2}).
\end{itemize}

\noindent
{\bf Strategies for proving boundedness and blow-up.}
The strategy for showing boundedness is to establish the differential inequality 
\begin{align}\label{bddgoal}
    \frac{d}{dt}\int_\Omega (u+1)^\sigma
     \le -c_1
     \Big(\int_\Omega (u+1)^\sigma\Big)^{1+\theta_1}
    +c_2
\end{align}
with some $\sigma>n$, $c_1, c_2, \theta_1>0$. 
The key to the derivation of \eqref{bddgoal} is to take advantage of 
the effect of repulsion. 
More precisely, we will estimate positive terms like $\chi\alpha \int_\Omega  u^{\sigma+p-2}$ 
by the negative term $-\xi\gamma \int_\Omega u^{\sigma+q-2}$. 
On the other hand, the cornerstone of the proof of finite-time blow-up is 
the derivation of the differential inequality
\begin{align}\label{bugoal}
\frac{\pa \phi}{\pa t}(s_0, t) \ge c_3s_0^{-\theta_2}\phi^2(s_0,t)
-c_4s_0^{\theta_3},
\end{align}
where $c_3, c_4, \theta_2, \theta_3>0$ are constants. 
Here the moment-type functional $\phi$ is defined as 
$\phi(s_0, t):=\int_0^{s_0}s^{-b}(s_0-s)U(s,t)\,ds$, 
where $U$ is the mass accumulation function given by 
$U(s,t):=\int_0^{s^{\frac1n}} \rho^{n-1}u(\rho, t)\,d\rho$
for $s>0$, $t>0$ and $b \in (0,1)$. 
To derive the inequality \eqref{bugoal} we utilize the attraction term. 
More precisely, the key is to handle a term derived from the repulsion term 
by exploiting the effect of attraction. 
\vspace{2.5mm}

\noindent
{\bf Plan of the paper.} 
This paper is organized as follows. In Section~\ref{Sec2} we collect some preliminary facts about local existence
in \eqref{prob} and a lemma guaranteeing an $L^\infty$-estimate from an 
$L^\sigma$-estimate for $u$ as well as an inequality which will be used later.
Section~\ref{Sec3} is devoted to establishing results on global existence and 
boundedness. 
In Section~\ref{Sec4} we give and prove results on finite-time blow-up. 


\section{Preliminaries} \label{Sec2}

We first give a result on local classical solutions to \eqref{prob}. 
This result can be proved by standard arguments 
based on the contraction mapping principle 
(see e.g., \cite{CW-2008, TW-2012, TW-2007}).
\begin{lem}\label{lem_local}
Let\/ $\Omega \subset \R^n$ $(n \in \N)$ 
be a bounded domain with smooth boundary and 
let $m, p, q \in \R$, $\chi, \xi, \alpha, \beta, \gamma, \delta>0$. 
Assume that $f(u) \equiv 0$ or $f(u)=\lambda u-\mu u^\kappa$ 
$(\kappa \ge 1)$, where $\lambda, \mu \in C^0(\cl{\Omega})$. 
Then for all $u_0$ satisfying the condition \eqref{u0} 
there exists $\tmax \in (0,\infty]$ such that \eqref{prob}
admits a unique classical solution $(u, v, w)$ such that 
\begin{align}\label{class}
\begin{cases}
       u \in C^0(\cl{\Omega} \times [0, \tmax)) \cap 
              C^{2,1}(\cl{\Omega} \times (0, \tmax)), 
     \\
       v, w \in 
        \bigcap_{\vartheta>n}C^0([0, \tmax); W^{1,\vartheta}(\Omega))
        \cap C^{2,1}(\cl{\Omega} \times (0, \tmax)).
\end{cases}
\end{align}
Moreover, 
\begin{align}\label{bc}
     {\it if}\ \tmax<\infty, \quad
     {\it then}\ 
     \lim_{t \nearrow \tmax}\|u(\cdot, t)\|_{L^\infty(\Omega)}=\infty.
\end{align}
Particularly, 
in the case that $f(u)=\lambda(|x|) u-\mu(|x|) u^\kappa$ 
$(\kappa \ge 1)$ and the conditions 
\eqref{omega}, \eqref{lammu} hold, 
if $u_0$ is further assumed to be radially symmetric, 
then there exists $\tmax \in (0,\infty]$ such that 
\eqref{prob} possesses a unique radially symmetric 
classical solution $(u, v, w)$ satisfying \eqref{class} and \eqref{bc}. 
\end{lem}

We next give the following lemma which provides a strategy 
to prove global existence and boundedness.

\begin{lem}\label{lem_Lsig_Linf}
Let\/ $\Omega \subset \R^n$ $(n \in \N)$ 
be a bounded domain with smooth boundary and 
let $m, p, q \in \R$, $\chi, \xi, \alpha, \beta, \gamma, \delta>0$. 
Assume that $f(u) \equiv 0$ and 
$u_0$ satisfies \eqref{u0}.
Denote by $(u, v, w)$ the local classical solution of \eqref{prob} 
given in Lemma~{\rm \ref{lem_local}} 
and by $\tmax \in (0,\infty]$ its maximal existence time. 
If for some $\sigma>n$,  
\begin{align*}
\sup_{t \in (0, \tmax)}\|u(\cdot, t)\|_{L^\sigma(\Omega)}<\infty, 
\end{align*}
then we have
\begin{align}\label{Linfty}
\sup_{t \in (0, \tmax)}\|u(\cdot, t)\|_{L^\infty(\Omega)}<\infty. 
\end{align}
\end{lem}

\begin{proof}
By the $L^\sigma$-boundedness of $u$, 
there exist $c_1>0$ and $\sigma>n$ such that 
\begin{align}\label{Lsigma2}
\|u(\cdot,t)\|_{L^\sigma(\Omega)} \le c_1
\end{align}
for all $t \in (0, \tmax)$. 
Since $\sigma>n$, 
applying \cite[Lemma~2.4~(ii) with $\theta=\sigma$ and $\mu=\infty$]{W-2011} along with \eqref{Lsigma2} yields  
\begin{align}
\|\nabla v(\cdot,t)\|_{L^\infty(\Omega)} 
&\le c_2\Big(1+\sup_{t \in (0, \tmax)}\|u(\cdot, t)\|_{L^\sigma(\Omega)}\Big)
\le c_3, \label{nabvest}\\
\|\nabla w(\cdot,t)\|_{L^\infty(\Omega)}
&\le c_4\Big(1+\sup_{t \in (0, \tmax)}\|u(\cdot, t)\|_{L^\sigma(\Omega)}\Big)
\le c_5 \label{nabwest}
\end{align}
for all $t \in (0, \tmax)$ with some $c_2, c_3, c_4, c_5>0$. 
Thanks to \eqref{Lsigma2}--\eqref{nabwest}, 
we can see from \cite[Lemma~A.1]{TW-2012} that \eqref{Linfty} holds. 
\end{proof}

We finally state an inequality which will be used repeatedly.

\begin{lem}\label{lem_some_ineq}
Let $\ell>1$.  Then for all $\ep>0$,
\begin{align}\label{conv}
(x+1)^\ell \le (1+\ep)x^\ell+C_\ep
\quad (x\ge0), 
\end{align}
where $C_\ep:=(1+\ep)\big((1+\ep)^{\frac{1}{\ell-1}}-1\big)^{-(\ell-1)}$.
\end{lem}

\begin{proof}
Owing to the convexity of the function $y \mapsto y^\ell$ on $[1,\infty)$ 
we have
\begin{align*}
  (x+1)^\ell
&=\left[
    \frac{1}{(1+\ep)^{\frac{1}{\ell-1}}}\cdot (1+\ep)^{\frac{1}{\ell-1}}x
    +\left(
         1-\frac{1}{(1+\ep)^{\frac{1}{\ell-1}}}
     \right) \cdot
         \frac{(1+\ep)^{{\frac{1}{\ell-1}}}}{(1+\ep)^{\frac{1}{\ell-1}}-1}
  \right]^\ell
  \\
&\le \frac{1}{(1+\ep)^{\frac{1}{\ell-1}}}\cdot 
     \left[
        (1+\ep)^{\frac{1}{\ell-1}}x
     \right]^\ell
    +\left(
         1-\frac{1}{(1+\ep)^{\frac{1}{\ell-1}}}
     \right) \cdot
     \left[
         \frac{(1+\ep)^{{\frac{1}{\ell-1}}}}{(1+\ep)^{\frac{1}{\ell-1}}-1}
     \right]^\ell
  \\
&=(1+\ep)x^\ell
    +\frac{1+\ep}{\big((1+\ep)^{\frac{1}{\ell-1}}-1\big)^{\ell-1}}, 
\end{align*}
which leads to \eqref{conv}.
\end{proof}


\section{Global existence and boundedness}
\label{Sec3}

In this section we assume that $\Omega \subset \R^n$ ($n \in \N$) is a bounded domain 
with smooth boundary, 
$m, p, q \in \R$, $\chi, \xi, \alpha, \beta, \gamma, \delta>0$, $f(u) \equiv 0$. 
We will prove global existence and boundedness in \eqref{prob} 
in two cases $p<q$ and $p=q$.


\subsection{The case \boldmath{$p<q$}}
\label{Subsec3.1}

In this subsection we show the following theorem 
which asserts global existence and boundedness 
in \eqref{prob} in the case $p<q$.

\begin{thm}\label{thm_bdd1}
Assume that $p<q$. 
Then for all $u_0$ satisfying \eqref{u0} 
there exists a unique triplet $(u, v, w)$ 
of nonnegative functions
\begin{align*}
\begin{cases}
       u \in C^0(\cl{\Omega} \times [0, \infty)) \cap 
              C^{2,1}(\cl{\Omega} \times (0, \infty)), 
     \\
       v, w \in 
        \bigcap_{\vartheta>n}C^0([0, \infty); W^{1,\vartheta}(\Omega))
        \cap C^{2,1}(\cl{\Omega} \times (0, \infty)),
\end{cases}
\end{align*}
which solves \eqref{prob} in the classical sense, 
and is bounded, that is, $\|u(\cdot,t)\|_{L^\infty(\Omega)} \le C$ 
for all $t>0$ with some $C>0$.
\end{thm}

In the following we denote by $(u, v, w)$ the local classical solution of \eqref{prob} 
given in Lemma~{\rm \ref{lem_local}} 
and by $\tmax \in (0,\infty]$ its maximal existence time. 
To prove Theorem~\ref{thm_bdd1}, 
it is sufficient to derive 
an $L^\sigma$-estimate for $u$ with some 
$\sigma>n$,  
because Lemma~\ref{lem_Lsig_Linf} leads to an 
$L^\infty$-estimate for $u$ which together with the criterion \eqref{bc} 
implies the conclusion. 
The following lemma plays an important role in the derivation of the $L^\sigma$-estimate.
\begin{lem}\label{lemsub}
Let $\ell>1$. 
Then the first and third components of the solution satisfy that 
for all $\ep>0$, 
\begin{align*}
\int_\Omega w^{\ell} \le \ep\int_\Omega u^{\ell}+c(\ep)\quad {\it on}\ (0, \tmax)
\end{align*}
with some $c(\ep)>0$. 
\end{lem}

\begin{proof}
Let $t \in (0, \tmax)$ and put 
$u:=u(\cdot, t)$, $w:=w(\cdot, t)$. 
Multiplying the third equation in \eqref{prob} by $w^{\ell-1}$ 
and integrating it over $\Omega$, we obtain 
\begin{align*}
         \delta \int_\Omega w^{\ell}-\int_\Omega w^{\ell-1}\Delta w
    &= \gamma \int_\Omega uw^{\ell-1}.
\end{align*}
Since the second term on the left-hand side is rewritten as 
\begin{align*}
   -\int_\Omega w^{\ell-1}\Delta w
&= (\ell-1)\int_\Omega w^{\ell-2}|\nabla w|^2 =\frac{4(\ell-1)}{\ell^2}
    \int_\Omega\big|\nabla w^{\frac{\ell}{2}}\big|^2, 
\end{align*}
we infer
\begin{align}\label{uw2}
         \delta \int_\Omega w^{\ell}
         +\frac{4(\ell-1)}{\ell^2}
    \int_\Omega\big|\nabla w^{\frac{\ell}{2}}\big|^2
    &=\gamma \int_\Omega uw^{\ell-1}.
\end{align}
Here we note from the first equation in \eqref{prob} 
that the mass conservation 
$\int_\Omega u(\cdot, t)=\int_\Omega u_0$ holds 
for all $t \in (0, \tmax)$. 
Hence, integrating the third equation in \eqref{prob} over $\Omega$ gives 
\begin{align}\label{mass}
\int_\Omega w=\frac{\gamma}{\delta}\int_\Omega u=\frac{\gamma}{\delta}\int_\Omega u_0.
\end{align}
Applying the Gagliardo--Nirenberg inequality to $\big\|w^{\frac{\ell}{2}}\big\|_{L^2(\Omega)}$ 
and using the relation \eqref{mass}, we see that there exist $c_1, c_2>0$ such that
\begin{align}\label{GN}
       \big\|w^{\frac{\ell}{2}}\big\|_{L^2(\Omega)}
&\le c_1\Big(\big\|
            \nabla w^{\frac{\ell}{2}}
            \big\|_{L^2(\Omega)}^{\theta_1}
            \big\|w^{\frac{\ell}{2}}
            \big\|_{L^{\frac{2}{\ell}}(\Omega)}^{1-\theta_1}
       +\big\|w^{\frac{\ell}{2}}
            \big\|_{L^{\frac{2}{\ell}}(\Omega)}\Big) \notag
\\
&\le c_2\Big(
    \big\|\nabla w^{\frac{\ell}{2}}\big\|_{L^2(\Omega)}^{\theta_1}+1
    \Big),
\end{align}
where 
$\theta_1:=\frac{\frac{\ell}{2}-\frac{1}{2}}
                       {\frac{\ell}{2}+\frac{1}{n}-\frac{1}{2}} \in (0,1)$. 
Let $\ep>0$ (fixed later). 
Then Young's inequality implies that there exists $c_3(\ep)>0$ such that 
\begin{align*}
    \big\|\nabla w^{\frac{\ell}{2}}\big\|_{L^2(\Omega)}^{\theta_1}
\le \frac{1}{c_2}\sqrt{\frac{\ep}{2}}\big\|
                     \nabla w^{\frac{\ell}{2}}
                     \big\|_{L^2(\Omega)}
     +c_3(\ep).
\end{align*}
This together with \eqref{GN} yields that 
\begin{align*}
       \big\|w^{\frac{\ell}{2}}\big\|_{L^2(\Omega)}^2
       &\le \Big(\sqrt{\frac{\ep}{2}}\big\|
               \nabla w^{\frac{\ell}{2}}
               \big\|_{L^2(\Omega)} +c_2(c_3(\ep)+1)\Big)^2\\
&\le 
\ep\big\|
               \nabla w^{\frac{\ell}{2}}
               \big\|_{L^2(\Omega)}^2 +c_4(\ep)
\end{align*}
with some $c_4(\ep)>0$. 
Namely, we have
\begin{align}\label{west}
\int_\Omega\big|\nabla w^{\frac{\ell}{2}}\big|^2
\ge \frac{1}{\ep}\int_\Omega w^{\ell}-c_5(\ep)
\end{align}
with some $c_5(\ep)>0$. 
Combining \eqref{uw2} with \eqref{west} and using 
H$\ddot{{\rm o}}$lder's and Young's inequalities, we derive that  
\begin{align*}
\delta \int_\Omega w^{\ell}
   +\frac{c_6}{\ep}
    \int_\Omega w^{\ell} 
&\le \gamma \int_\Omega uw^{\ell-1}
	       +c_7(\ep)  \notag  \\
&\le \gamma 
       \Big(\int_\Omega u^{\ell}\Big)^{\frac{1}{\ell}}
       \Big(\int_\Omega w^{\ell}\Big)^{\frac{\ell-1}{\ell}}
       +c_7(\ep)  \notag  \\
&\le \gamma\Big[\frac{1}{\ell}\int_\Omega u^{\ell}
       +\Big(1-\frac{1}{\ell}\Big)\int_\Omega w^{\ell}\Big]
       +c_7(\ep)
\end{align*}
with some $c_6, c_7(\ep)>0$, and thus infer
\begin{align}\label{west2}
     \Big(\delta+\frac{c_6}{\ep}-\gamma+\frac{\gamma}{\ell}\Big)
     \int_\Omega w^{\ell} 
\le \frac{\gamma}{\ell}\int_\Omega u^{\ell}+c_7(\ep).
\end{align}
We now observe that if $\ep \in (0, \frac{c_6}{\gamma})$ then 
$\frac{c_6}{\ep}-\gamma>0$, that is,
\begin{align*}
\delta+\frac{c_6}{\ep}-\gamma+\frac{\gamma}{\ell}>0. 
\end{align*}
Therefore, picking $\ep \in (0, \frac{c_6}{\gamma})$, 
we have from \eqref{west2} that 
\begin{align*}
\int_\Omega w^{\ell} &\le \frac{\frac{\gamma}{\ell}}{\delta+\frac{c_6}{\ep}-\gamma+\frac{\gamma}{\ell}}
\int_\Omega u^{\ell}+
\frac{c_7(\ep)}{\delta+\frac{c_6}{\ep}-\gamma+\frac{\gamma}{\ell}}\notag\\
&= \frac{\frac{\gamma}{\ell}\ep}{(\delta-\gamma+\frac{\gamma}{\ell})\ep+c_6}
\int_\Omega u^{\ell}+
\frac{c_7(\ep)\ep}{(\delta-\gamma+\frac{\gamma}{\ell})\ep+c_6}.
\end{align*}
Noticing that for all $\overline{\ep}>0$ there exists 
$\ep \in (0, \frac{c_6}{\gamma})$ such that 
$\frac{\frac{\gamma}{\ell}\ep}{(\delta-\gamma+\frac{\gamma}{\ell})\ep+c_6}<\overline{\ep}$, 
we arrive at the conclusion. 
\end{proof}

We now prove an $L^\sigma$-estimate for $u$. 

\begin{lem}\label{lem_Lsig_p<q}
Assume that $p<q$. 
Then for some $\sigma>n$ there exists $C>0$ such that 
\begin{align*}
    \|u(\cdot, t)\|_{L^\sigma(\Omega)} \le C
\end{align*}
for all $t \in (0,\tmax)$.
\end{lem}

\begin{proof}
Let $\sigma>1$ be sufficiently large. 
We first obtain from the first equation in \eqref{prob} with $f(u) \equiv 0$ and 
integration by parts that 
\begin{align}\label{DI1}
  &\frac{1}{\sigma}\frac{d}{dt}\int_\Omega (u+1)^\sigma \notag 
\\
  &\quad\,
  =\int_\Omega (u+1)^{\sigma-1}\nabla \cdot
                 \big((u+1)^{m-1}\nabla u\big) \notag
\\
  &\qquad\,\,
   -\chi\int_\Omega (u+1)^{\sigma-1}\nabla\cdot
   \big(u(u+1)^{p-2}\nabla v\big)
   +\xi\int_\Omega (u+1)^{\sigma-1}\nabla\cdot
   \big(u(u+1)^{q-2}\nabla w\big) \notag
\\
  &\quad\,=
  -(\sigma-1)\int_\Omega 
                 (u+1)^{\sigma+m-3}|\nabla u|^2\notag
\\
  &\qquad\,\,
   +\chi(\sigma-1)\int_\Omega 
                  u(u+1)^{\sigma+p-4}\nabla u\cdot\nabla v
   -\xi(\sigma-1)\int_\Omega 
                 u(u+1)^{\sigma+q-4}\nabla u\cdot\nabla w
   \notag
\\
  &\quad\,=:
   I_1+I_2+I_3
\end{align}
for all $t \in (0,\tmax)$. 
We estimate the terms $I_1, I_2, I_3$. 
As to the first term $I_1$, we rewrite it as 
\begin{align}\label{I1}
I_1=-\frac{4(\sigma-1)}{(\sigma+m-1)^2}
     \int_\Omega \big|\nabla(u+1)^{\frac{\sigma+m-1}{2}}\big|^2.
\end{align}
We next deal with the second term $I_2$ and third term $I_3$. 
As to the former, integration by parts and the second equation in \eqref{prob} 
lead to 
\begin{align}\label{I21}
I_2 &=\chi(\sigma-1)
        \int_\Omega \nabla \Big[\int_0^u s(s+1)^{\sigma+p-4}\,ds\Big]
                           \cdot \nabla v
    \notag
\\
    &=\chi(\sigma-1)
        \int_\Omega \Big[\int_0^u s(s+1)^{\sigma+p-4}\,ds\Big] \cdot(-\Delta v)
    \notag
\\
    &=\chi(\sigma-1)
        \int_\Omega \Big[\int_0^u s(s+1)^{\sigma+p-4}\,ds\Big]\cdot 
        (\alpha u-\beta v)
    \notag
\\
    &\le \chi\alpha(\sigma-1)
           \int_\Omega \Big[\int_0^u s(s+1)^{\sigma+p-4}\,ds\Big] u.
\end{align}
Here we infer that for $\sigma>-p+2$, 
\begin{align*}
    \Big[\int_0^u s(s+1)^{\sigma+p-4}\,ds\Big]u
&\le \Big[\int_0^u (s+1)^{\sigma+p-3}\,ds\Big]u \\
&\le \frac{1}{\sigma+p-2}(u+1)^{\sigma+p-2}u\\
&\le \frac{1}{\sigma+p-2}(u+1)^{\sigma+p-1}.
\end{align*}
Combining the above estimate with \eqref{I21} and 
using Lemma~\ref{lem_some_ineq} with $\ep=1$ and $\sigma>-p+2$, we have
\begin{align}\label{I22}
I_2 &\le \frac{\chi\alpha(\sigma-1)}{\sigma+p-2}
            \Big(2\int_\Omega u^{\sigma+p-1}+c_1\Big), 
\end{align}
with some $c_1>0$. 
Similarly, as to the term $I_3$, we establish
\begin{align}\label{I31}
I_3 &=\xi(\sigma-1)
        \int_\Omega \Big[\int_0^u s(s+1)^{\sigma+q-4}\,ds\Big] \cdot\Delta w\notag \\
&=\xi(\sigma-1)
        \int_\Omega \Big[\int_0^u s(s+1)^{\sigma+q-4}\,ds\Big] \cdot (\delta w-\gamma u).
\end{align}
Here, noting that $s^{\sigma+q-3} \le s(s+1)^{\sigma+q-4} \le (s+1)^{\sigma+q-3}$ for $\sigma\ge-q+4$, 
we see that
\begin{align}\label{uplow}
\frac{1}{\sigma+q-2}u^{\sigma+q-2} 
\le \int_0^u s(s+1)^{\sigma+q-4}\,ds
\le \frac{1}{\sigma+q-2}(u+1)^{\sigma+q-2},
\end{align}
where we neglected the term 
$-\frac{1}{\sigma+q-2}$ 
on the most right-hand side. 
Due to Lemma~\ref{lem_some_ineq} with $\ep=1$ we obtain that
\begin{align}\label{uw01}
       \Big[\int_0^u s(s+1)^{\sigma+q-4}\,ds\Big]w
&\le \frac{1}{\sigma+q-2} (u+1)^{\sigma+q-2}w \notag \\
&\le \frac{1}{\sigma+q-2} \Big(2u^{\sigma+q-2}w+c_2w\Big), 
\end{align}
with some $c_2>0$. 
Therefore a combination of the above estimates 
\eqref{I31}--\eqref{uw01} yields that 
\begin{align}\label{I32}
I_3 &\le \frac{\xi(\sigma-1)}{\sigma+q-2}
            \Big(2\delta\int_\Omega u^{\sigma+q-2}w
            +\delta c_2\int_\Omega w
                   -\gamma\int_\Omega u^{\sigma+q-1}\Big).
\end{align}
Collecting \eqref{I1}, \eqref{I22} and \eqref{I32} 
in \eqref{DI1}, we derive
\begin{align}\label{I33}
    \frac{1}{\sigma}\frac{d}{dt}\int_\Omega (u+1)^\sigma &\le 
      -\frac{4(\sigma-1)}{(\sigma+m-1)^2}
     \int_\Omega \big|\nabla(u+1)^{\frac{\sigma+m-1}{2}}\big|^2 \notag
\\
    &\quad\,
      +\frac{\chi\alpha(\sigma-1)}{\sigma+p-2}
            \Big(2\int_\Omega u^{\sigma+p-1}+c_1\Big) \notag
\\
    &\quad\,
      +\frac{\xi(\sigma-1)}{\sigma+q-2}
            \Big(2\delta\int_\Omega u^{\sigma+q-2}w
            +\delta c_2\int_\Omega w
                   -\gamma\int_\Omega u^{\sigma+q-1}\Big)
\end{align}
for all $t \in (0,\tmax)$. 
Moreover, taking $\ep_1>0$ which will be fixed later and applying Young's inequality to $u^{\sigma+p-1}$, 
we have $u^{\sigma+p-1} \le \ep_1 u^{\sigma+q-1}+c_3(\ep_1)$ with some $c_3(\ep_1)>0$.
Additionally, again by the relation \eqref{mass} we see that
\begin{align}\label{DI2}
    &\frac{1}{\sigma}\frac{d}{dt}\int_\Omega (u+1)^\sigma
      +\frac{4(\sigma-1)}{(\sigma+m-1)^2}
     \int_\Omega \big|\nabla(u+1)^{\frac{\sigma+m-1}{2}}\big|^2 \notag 
\\ 
    &\quad\, \le 
      \frac{\chi\alpha(\sigma-1)}{\sigma+p-2}
            \Big[2\Big(\ep_1\int_\Omega u^{\sigma+q-1}
                   +c_3(\ep_1)\Big)+c_1\Big] \notag
\\
    &\qquad\,\,
      +\frac{\xi(\sigma-1)}{\sigma+q-2}
            \Big(2\delta\int_\Omega u^{\sigma+q-2}w
            +c_4
                   -\gamma\int_\Omega u^{\sigma+q-1}\Big)
\end{align}
for all $t \in (0,\tmax)$ 
with some $c_4>0$. 
We next estimate the term $\int_\Omega u^{\sigma+q-2}w$.  
Using the H$\ddot{{\rm o}}$lder inequality, we infer
\begin{align*}
     \int_\Omega u^{\sigma+q-2}w
\le \Big(\int_\Omega u^{\sigma+q-1}\Big)^{\frac{\sigma+q-2}{\sigma+q-1}}
     \Big(\int_\Omega w^{\sigma+q-1}\Big)^{\frac{1}{\sigma+q-1}}.
\end{align*}
Here we take $\ep_2>0$ which will be fixed later. 
Employing the Young inequality
as well as applying Lemma~\ref{lemsub} with $\ell=\sigma+q-1$ 
and $\ep=(\frac{\ep_2}{2})^{\sigma+q-1}$ to 
$\int_\Omega w^{\sigma+q-1}$, 
we establish
\begin{align}\label{uwest*}
  \int_\Omega u^{\sigma+q-2}w 
	&\le \Big(\int_\Omega u^{\sigma+q-1}\Big)^{\frac{\sigma+q-2}{\sigma+q-1}}
     \Big[\Big(\frac{\ep_2}{2}\Big)^{\sigma+q-1}\int_\Omega u^{\sigma+q-1}+c_5(\ep_2)\Big]^{\frac{1}{\sigma+q-1}}\notag\\
&\le \frac{\ep_2}{2}\int_\Omega u^{\sigma+q-1}+c_5(\ep_2)^{{\frac{1}{\sigma+q-1}}}\Big(\int_\Omega u^{\sigma+q-1}\Big)^{\frac{\sigma+q-2}{\sigma+q-1}} \notag\\
&\le \frac{\ep_2}{2}\int_\Omega u^{\sigma+q-1}+c_5(\ep_2)^{{\frac{1}{\sigma+q-1}}}\Big(\frac{\ep_2}{2c_5(\ep_2)^{{\frac{1}{\sigma+q-1}}}}\int_\Omega u^{{\sigma+q-1}}+c_6(\ep_2)\Big) \notag\\
&=\ep_2\int_\Omega u^{\sigma+q-1}+c_7(\ep_2)
\end{align}
with some $c_5(\ep_2), c_6(\ep_2), c_7(\ep_2)>0$. 
Setting 
$c_8:=\frac{\chi\alpha(\sigma-1)}{\sigma+p-2}$ and
$c_9:=\frac{\xi(\sigma-1)}{\sigma+q-2}$, 
we derive from \eqref{DI2} and \eqref{uwest*} that
\begin{align}\label{uest}
    &\frac{1}{\sigma}\frac{d}{dt}\int_\Omega (u+1)^\sigma
      +\frac{4(\sigma-1)}{(\sigma+m-1)^2}
     \int_\Omega \big|\nabla(u+1)^{\frac{\sigma+m-1}{2}}\big|^2 \notag 
\\ 
    &\quad\, \le           \frac{\chi\alpha(\sigma-1)}{\sigma+p-2}
            \Big[2\Big(\ep_1\int_\Omega u^{\sigma+q-1}
                   +c_3(\ep_1)\Big)+c_1\Big] \notag
\\
    &\qquad\,\,
      +\frac{\xi(\sigma-1)}{\sigma+q-2}
            \Big(2\delta\int_\Omega u^{\sigma+q-2}w
            +c_4
                   -\gamma\int_\Omega u^{\sigma+q-1}\Big) \notag\\
&\quad\, \le 2c_8\ep_1\int_\Omega u^{\sigma+q-1}
+c_9\Big[2\delta\Big(\ep_2\int_\Omega u^{\sigma+q-1}+c_7(\ep_2)\Big)
-\gamma\int_\Omega u^{\sigma+q-1}\Big]+c_{10}(\ep_1) \notag\\
&\quad\,= 2c_8\ep_1\int_\Omega u^{\sigma+q-1}
+c_9(2\delta\ep_2-\gamma)\int_\Omega u^{\sigma+q-1}+c_{11}(\ep_1, \ep_2)
\end{align}
for all $t \in (0,\tmax)$ with some $c_{10}(\ep_1), c_{11}(\ep_1, \ep_2)>0$. 
Here we choose $\ep_2>0$ satisfying $\ep_2<\frac{\gamma}{2\delta}$, that is, 
$2\delta\ep_2-\gamma<0$. 
Then we have from \eqref{uest} that
\begin{align}\label{uest*}
&\frac{1}{\sigma}\frac{d}{dt}\int_\Omega (u+1)^\sigma
      +\frac{4(\sigma-1)}{(\sigma+m-1)^2}
     \int_\Omega \big|\nabla(u+1)^{\frac{\sigma+m-1}{2}}\big|^2 \notag 
\\ 
&\quad\, \le \big(2c_8\ep_1
-c_9(\gamma-2\delta\ep_2)\big)\int_\Omega u^{\sigma+q-1}+
c_{11}(\ep_1)
\end{align}
for all $t \in (0,\tmax)$. 
We let 
\begin{align*}
\ep_1:=\frac{c_9(\gamma-2\delta\ep_2)}{2c_8}>0.
\end{align*}
Therefore we obtain from \eqref{uest*} that
\begin{align}\label{DI3}
    \frac{1}{\sigma}\frac{d}{dt}\int_\Omega (u+1)^\sigma
      +\frac{4(\sigma-1)}{(\sigma+m-1)^2}
     \int_\Omega \big|\nabla(u+1)^{\frac{\sigma+m-1}{2}}\big|^2 \le c_{11}
\end{align}
for all $t \in (0,\tmax)$. 
We finally estimate the second term on the left-hand side of \eqref{DI3} 
in order to derive a differential inequality for $\int_\Omega (u+1)^{\sigma}$. 
Again using the Gagliardo--Nirenberg inequality and the mass conservation, 
we see that
\begin{align*}
&\|u(\cdot,t)+1\|_{L^\sigma(\Omega)}^\sigma\\
&\quad\,
=\big\|(u(\cdot,t)+1)^{\frac{\sigma+m-1}{2}}\big\|_{L^{\frac{2\sigma}{\sigma+m-1}}(\Omega)}^{\frac{2}{\sigma+m-1}}\\
&\quad\,\le c_{12}\Big(\big\|\nabla(u(\cdot,t)+1)^{\frac{\sigma+m-1}{2}}\big\|_{L^2(\Omega)}^{\theta_2}\big\|(u(\cdot,t)+1)^{\frac{\sigma+m-1}{2}}\big\|_{L^{\frac{2}{\sigma+m-1}}(\Omega)}^{1-\theta_2}\\
&\hspace{7cm}
+\big\|(u(\cdot,t)+1)^{\frac{\sigma+m-1}{2}}\big\|_{L^{\frac{2}{\sigma+m-1}}(\Omega)}\Big)^{\frac{2}{\sigma+m-1}}\\
&\quad\,\le c_{12}\Big(\big\|\nabla(u(\cdot,t)+1)^{\frac{\sigma+m-1}{2}}\big\|_{L^2(\Omega)}^{\frac{2}{\sigma+m-1}\theta_2}\|u(\cdot,t)+1\|_{L^1(\Omega)}^{1-\theta_2}
+\|u(\cdot,t)+1\|_{L^1(\Omega)}\Big)\\
&\quad\,
\le c_{13}\Big(\big\|\nabla(u(\cdot,t)+1)^{\frac{\sigma+m-1}{2}}\big\|_{L^2(\Omega)}^{\frac{2}{\sigma+m-1}\theta_2}+1\Big)
\end{align*}
for all $t \in (0,\tmax)$ with $\theta_2:=\frac{\frac{\sigma+m-1}{2}-\frac{\sigma+m-1}{2\sigma}}{\frac{\sigma+m-1}{2}+\frac1n-\frac12} \in (0,1)$ 
and $c_{12}, c_{13}>0$. 
This implies 
\begin{align}\label{GNu}
\big\|\nabla(u(\cdot,t)+1)^{\frac{\sigma+m-1}{2}}\big\|_{L^2(\Omega)} 
&\ge \Big(\frac{1}{c_{13}}\|u(\cdot,t)+1\|_{L^\sigma(\Omega)}^\sigma-1\Big)^{\frac{\sigma+m-1}{2\theta_2}} \notag \\
&\ge c_{14}\|u(\cdot,t)+1\|_{L^\sigma(\Omega)}^{\frac{\sigma+m-1}{2\theta_2}}-1
\end{align}
for all $t \in (0,\tmax)$ with some $c_{14}>0$. 
A combination of \eqref{DI3} and \eqref{GNu} yields that
\begin{align*}
    \frac{1}{\sigma}\frac{d}{dt}\int_\Omega (u+1)^\sigma
    +c_{15}
     \Big(\int_\Omega (u+1)^\sigma\Big)^{\frac{\sigma+m-1}{2\theta_2}}
    &\le c_{16}
\end{align*}
for all $t \in (0,\tmax)$ with some $c_{15}, c_{16}>0$. 
Noting from $2\theta_2<2<\sigma+m-1$ for sufficiently large $\sigma$ 
that $\frac{\sigma+m-1}{2\theta_2}>1$, we infer that
\begin{align*}
\int_\Omega (u+1)^\sigma \le c_{17}
\end{align*}
with some $c_{17}>0$. 
This proves the conclusion for all sufficiently large $\sigma>1$. 
\end{proof}

We are now in a position to complete the proof of Theorem~\ref{thm_bdd1}.

\begin{prth3.1} 
A combination of Lemmas~\ref{lem_Lsig_p<q} and~\ref{lem_Lsig_Linf} along with the criterion 
\eqref{bc} leads to the conclusion of Theorem~\ref{thm_bdd1}. \qed
\end{prth3.1}


\subsection{The case \boldmath{$p=q$}}
\label{Subsec3.2}

In this subsection we state the following theorem 
guaranteeing global existence and boundedness 
in \eqref{prob} in the case $p=q$.

\begin{thm}\label{thm_bdd2}
Assume that $p=q$ and $\chi\alpha-\xi\gamma<0$. 
Then for all $u_0$ satisfying \eqref{u0} 
there exists a unique triplet $(u, v, w)$ 
of nonnegative functions
\begin{align*}
\begin{cases}
       u \in C^0(\cl{\Omega} \times [0, \infty)) \cap 
              C^{2,1}(\cl{\Omega} \times (0, \infty)), 
     \\
       v, w \in 
        \bigcap_{\vartheta>n}C^0([0, \infty); W^{1,\vartheta}(\Omega))
        \cap C^{2,1}(\cl{\Omega} \times (0, \infty)),
\end{cases}
\end{align*}
which solves \eqref{prob} in the classical sense, 
and is bounded, that is, $\|u(\cdot,t)\|_{L^\infty(\Omega)} \le C$ 
for all $t>0$ with some $C>0$.
\end{thm}

As in the previous subsection, 
we denote by $(u, v, w)$ the local classical solution of \eqref{prob} 
given in Lemma~{\rm \ref{lem_local}} 
and by $\tmax \in (0,\infty]$ its maximal existence time. 
We prove Theorem~\ref{thm_bdd2} by deriving an $L^\sigma$-estimate for $u$.

\begin{lem}\label{lem_Lsig_p=q}
Suppose that $p=q$. Then for some $\sigma>n$ there exists $C>0$ such that 
\begin{align*}
    \|u(\cdot, t)\|_{L^\sigma(\Omega)} \le C
\end{align*}
for all $t \in (0,\tmax)$.
\end{lem}

\begin{proof} 
Let $\sigma>1$ be sufficiently large. 
Let $\ep_1>0$ which will be fixed later. 
Proceeding similarly in the proof of Lemma~\ref{lem_Lsig_p<q}, 
we see that \eqref{I33} with $p=q$ holds, that is,
\begin{align*}
    \frac{1}{\sigma}\frac{d}{dt}\int_\Omega (u+1)^\sigma&\le 
      -\frac{4(\sigma-1)}{(\sigma+m-1)^2}
     \int_\Omega \big|\nabla(u+1)^{\frac{\sigma+m-1}{2}}\big|^2 \notag
\\
    &\quad\,
      +\frac{\chi\alpha(\sigma-1)}{\sigma+p-2}
            \Big((1+\ep_1)\int_\Omega u^{\sigma+p-1}+c_1(\ep_1)\Big) \notag
\\
    &\quad\,
      +\frac{\xi(\sigma-1)}{\sigma+p-2}
            \Big(2\delta\int_\Omega u^{\sigma+p-2}w
            +\delta c_2\int_\Omega w
                   -\gamma\int_\Omega u^{\sigma+p-1}\Big)
\end{align*}
for all $t \in (0,\tmax)$ with some $c_1(\ep_1), c_2>0$.  
Also, setting $c_3:=\frac{\sigma-1}{\sigma+p-2}$ and 
recalling the property \eqref{mass}, we have
\begin{align}\label{2DI2}
    &\frac{1}{\sigma}\frac{d}{dt}\int_\Omega (u+1)^\sigma
      +\frac{4(\sigma-1)}{(\sigma+m-1)^2}
     \int_\Omega \big|\nabla(u+1)^{\frac{\sigma+m-1}{2}}\big|^2 \notag 
\\ 
    &\quad\,\,\le
      \chi\alpha c_3
            \Big((1+\ep_1)\int_\Omega u^{\sigma+p-1}+c_1(\ep_1)\Big) \notag
\\
    &\qquad\,\,
      +\xi c_3
            \Big(2\delta\int_\Omega u^{\sigma+p-2}w
            +c_4
                   -\gamma\int_\Omega u^{\sigma+p-1}\Big),
\end{align}
for all $t \in (0,\tmax)$ with some $c_4>0$. 
We now take $\ep_2>0$ which will be fixed later. 
Then, an argument similar to that in derivation of \eqref{uwest*} implies
\begin{align*}
  \int_\Omega u^{\sigma+p-2}w 
	\le\frac{\ep_2}{2\xi\delta}\int_\Omega u^{\sigma+p-1}+c_5(\ep_2)
\end{align*}
with some $c_5(\ep_2)>0$. 
Thus we obtain
\begin{align}\label{2DI3}
    &\frac{1}{\sigma}\frac{d}{dt}\int_\Omega (u+1)^\sigma
      +\frac{4(\sigma-1)}{(\sigma+m-1)^2}
     \int_\Omega \big|\nabla(u+1)^{\frac{\sigma+m-1}{2}}\big|^2 \notag 
\\ 
    &\quad\,\,\le
      \chi\alpha c_3
            \Big((1+\ep_1)\int_\Omega u^{\sigma+p-1}+c_1(\ep_1)\Big) \notag
\\
    &\qquad\,\,
      +\xi c_3
            \Big(2\delta\int_\Omega u^{\sigma+p-2}w
            +c_4
                   -\gamma\int_\Omega u^{\sigma+p-1}\Big) \notag\\
&\quad\,\le 
c_3\Big[\chi\alpha(1+\ep_1)\int_\Omega u^{\sigma+p-1}
+2\xi\delta\Big(\frac{\ep_2}{2\xi\delta}\int_\Omega u^{\sigma+p-1}+c_5(\ep_2)\Big)
-\xi\gamma\int_\Omega u^{\sigma+p-1}\Big]\notag\\
&\qquad\,\,+c_6(\ep_1) \notag\\
&\quad\,=c_3\Big[\big(\chi\alpha(1+\ep_1)-\xi\gamma\big)+\ep_2\Big]\int_\Omega u^{\sigma+p-1}+c_7(\ep_1,\ep_2)
\end{align}
for all $t \in (0,\tmax)$ with some $c_6(\ep_1), c_7(\ep_1, \ep_2)>0$. 
Here since $\chi\alpha-\xi\gamma<0$ by assumption, we can pick $\ep_1>0$ 
satisfying $\chi\alpha(1+\ep_1)-\xi\gamma<0$. 
Then, taking
\begin{align*}
\ep_2:=\xi\gamma-\chi\alpha(1+\ep_1)>0, 
\end{align*}
we have from \eqref{2DI2} and \eqref{2DI3} that
\begin{align*}
\frac{1}{\sigma}\frac{d}{dt}\int_\Omega (u+1)^\sigma
      +\frac{4(\sigma-1)}{(\sigma+m-1)^2}
     \int_\Omega \big|\nabla(u+1)^{\frac{\sigma+m-1}{2}}\big|^2 \le c_7
\end{align*}
for all $t \in (0,\tmax)$. 
Finally, deriving a differential inequality for $\int_\Omega (u+1)^\sigma$ 
by an argument similar to that in the proof of Lemma~\ref{lem_Lsig_p<q}, 
we arrive at the conclusion. 
\end{proof}

Employing Lemma~\ref{lem_Lsig_p=q}, we can prove Theorem~\ref{thm_bdd2}.

\begin{prth3.3}
In view of Lemmas~\ref{lem_Lsig_p=q} and~\ref{lem_Lsig_Linf} along with the criterion \eqref{bc}, we immediately arrive at
the conclusion of Theorem~\ref{thm_bdd2}. \qed
\end{prth3.3}


\section{Finite-time blow-up} 
\label{Sec4}

In the following we suppose that $\Omega=B_R(0) \subset \R^n$ 
($n \in \N$, $n \ge 3$) with $R>0$ and 
$f(u)=\lambda(|x|)u-\mu(|x|)u^\kappa$ ($\kappa \ge 1$), 
where $\lambda, \mu$ satisfy the conditions \eqref{lammu} and \eqref{mupro} as well as 
$m>0$, $p, q \in \R$, $\chi, \xi, \alpha, \beta, \gamma, \delta>0$. 
We also assume that $u_0$ is radially symmetric and fulfills \eqref{u0}. 
Then we denote by $(u, v, w)=(u(r,t), v(r,t), w(r,t))$ the local classical solution of \eqref{prob} 
given in Lemma~\ref{lem_local} and by $\tmax \in (0, \infty]$ its maximal existence time. 

In order to state the main theorems we give the conditions 
\ref{C1}--\ref{C3} as follows:
         \begin{align}
         &\begin{cases}
         n \in \{3,4\};\quad m \ge 1,\quad 
         p<\dfrac{2}{n+1}m+\dfrac{2(n^2+1)}{n(n+1)}, 
         \notag
      \\[4mm]
         p<-\dfrac{1}{n-2}m+\dfrac{2(n^2-n-1)}{n(n-2)},\quad 
         m-p<-\dfrac{2}{n};
         \tag*{({\bf C1})}\label{C1}
         \end{cases}\\[6mm]
         &\begin{cases}
         n \ge 5;\quad m \ge 1,\quad
         -\dfrac{2}{n-3}m+\dfrac{2(n^2-2n-1)}{n(n-3)}
         <p<\dfrac{2}{n+1}m+\dfrac{2(n^2+1)}{n(n+1)},
         \notag
      \\[4mm]
         p<-\dfrac{n+2}{n-4}m+\dfrac{3n^2-5n-4}{n(n-4)},\quad 
         p\le \dfrac{n+2}{3}m-\dfrac{n^2-3n-4}{3n};
         \tag*{({\bf C2})}\label{C2}
         \end{cases}
      \\[6mm]
         &\begin{cases}
         n \ge 5;\quad m \ge 1,\quad
         -\dfrac{2}{n-3}m+\dfrac{2(n^2-2n-1)}{n(n-3)}
         <p<\dfrac{2}{n+1}m+\dfrac{2(n^2+1)}{n(n+1)},
         \notag
      \\[4mm]
         -\dfrac{n+2}{n-4}m+\dfrac{3n^2-5n-4}{n(n-4)} 
         \le p<-\dfrac{1}{n-2}m+\dfrac{2(n^2-n-1)}{n(n-2)},
      \\[4mm]
         m-p<-\dfrac{2}{n}.
         \tag*{({\bf C3})}\label{C3}
         \end{cases}
         \end{align}

%

%


\subsection{The case \boldmath{$p>q$}}
\label{Subsec4.1}

In this subsection we establish finite-time blow-up
in \eqref{prob} in the case $p>q$.

\begin{thm}\label{thm_bu1}
Assume that $p>q$. 
Also, suppose that $m>0$, $\kappa \ge 1$ fulfill the following conditions\/{\rm :}
\begin{itemize}
\item[{\rm (i)}] In the case {\rm \ref{C1}}, 
\begin{align*}
\kappa<1+\frac{(n-2)\big((m-p+1)n+1\big)}{n(n-1)}+\frac{a\big((m-p+1)n+1\big)}{n(n-1)}-(m-1)-(2-p)_+;
\end{align*}
\item[{\rm (ii)}] In the case {\rm \ref{C2}}, 
\begin{align*}
\kappa<1+\frac{(n-2)\big((m-p+1)n+1\big)}{n(n-1)}+\frac{a\big((m-p+1)n+1\big)}{n(n-1)}-(m-1)-(2-p)_+;
\end{align*}
\item[{\rm (iii)}] In the case {\rm \ref{C3}}, 
\begin{align*}
\kappa<1+\frac{(m-p+1)n+1}{2(n-1)}+\frac{a\big((m-p+1)n+1\big)}{n(n-1)}-\frac{(2-p)_+}{2},
\end{align*}
\end{itemize}
where $a\ge0$ is given in \eqref{mupro} and 
$y_+:=\max\{0, y\}$. 
Let $M_0>0$, $M_1 \in (0, M_0)$ and $L>0$. 
Then one can find $\ep_0>0$ and $r_1 \in (0, R)$ with the following property\/{\rm :} 
If $u_0$ satisfies 
$u_0(x) \le L|x|^{-\sigma}$, where 
$\sigma=\frac{n(n-1)}{(m-p+1)n+1}+\ep_0$ as well as 
$\int_\Omega u_0=M_0$ and $\int_{B_{r_1}(0)}u_0 \ge M_1$, 
then the solution $(u, v, w)$ to \eqref{prob} blows up at $t=T^* \in (0, \infty)$ 
in the sense that 
\begin{align*}
\lim_{t \nearrow T^*}\|u(\cdot, t)\|_{L^\infty(\Omega)}=\infty.
\end{align*}
\end{thm}

We first show the following lemma giving the profile of $u$, in which we include the case $p=q$ toward the next subsection. 

\begin{lem}\label{lem_profile}
Assume that $p \ge q$. 
Also, suppose that $m>0$ and $p>1$ fulfill
\begin{align*}
m \ge 1,\quad m-p\in \Big(-1-\frac{1}{n},\ -\frac{2}{n}\Big].
\end{align*}
Let $M_0>0$, $L>0$ and $T>0$. 
Let $\ep>0$ and set $\sigma:=\frac{n(n-1)}{(m-p+1)n+1}+\ep$. 
Then there exists $C>0$ such that the following property holds\/{\rm :} 
If $u_0$ satisfies $\int_\Omega u_0=M_0$ and
\begin{align*}
u_0(x) \le L|x|^{-\sigma}
\end{align*}
for all $x \in \Omega$, then the classical solution $(u, v, w) \in \big(C^0(\cl{\Omega} \times [0,T)) \cap C^{2,1}(\cl{\Omega} \times (0,T))\big)^3$ of \eqref{prob} has the estimate  
\begin{align}\label{profile2}
u(x,t) \le C|x|^{-\sigma}
\end{align}
for all $x \in \Omega$ and all $t \in (0,T)$. 
\end{lem}

\begin{proof}
By the condition for the function $\lambda$ (see \eqref{lammu}), 
we see that there exists $\lambda_1>0$ 
such that $\lambda(|x|) \le \lambda_1$ for all $x \in \Omega$. 
We next set
\begin{align*}
\widetilde{u}(x,t)&:=e^{-\lambda_1t}u(x,t),\quad 
D(x,t,\rho):=(e^{\lambda_1t}\rho+1)^{m-1},\\
S_1(x,t,\rho)&:=-\chi(e^{\lambda_1t}\rho+1)^{p-2}\rho,\quad
S_2(x,t,\rho):=\xi(e^{\lambda_1t}\rho+1)^{q-2}\rho
\end{align*}
for $x \in \Omega$, $t \in (0,T)$ and $\rho>0$. 
Since $S_1(\cdot, \cdot, \cdot)<0$ on $\Omega \times (0,T) \times (0,\infty)$, we have 
\begin{align*}
S_1(x,t,\rho)\nabla v(x,t)+S_2(x,t,\rho)\nabla w(x,t)
&=S_1(x,t,\rho)\Big[\nabla v(x,t)+\frac{S_2(x,t,\rho)}{S_1(x,t,\rho)}\nabla w(x,t)\Big]
\end{align*}
for all $x \in \Omega$, $t \in (0,T)$ and all $\rho>0$. 
Putting
\begin{align*}
\mathbf{f}(x,t):=\nabla v(x,t)+\frac{S_2(x,t,\rho)}{S_1(x,t,\rho)}\nabla w(x,t),
\end{align*}
we obtain from \eqref{prob} that
\begin{align}\label{P}
\begin{cases}
\widetilde{u}_t \le \nabla \cdot (D(x,t,\widetilde{u})\nabla\widetilde{u}
+S_1(x,t,\widetilde{u})\,\mathbf{f}(x,t))
& {\rm in}\ \Omega \times (0,T),\\
(D(x,t,\widetilde{u})\nabla\widetilde{u}
+S_1(x,t,\widetilde{u})\,\mathbf{f}(x,t)) \cdot \nu=0 
& {\rm on}\ \pa\Omega \times (0,T),\\
\widetilde{u}(\cdot,0)=u_0 & {\rm in}\ \Omega.
\end{cases}
\end{align}
Also, it can be checked that for all $x \in \Omega$, $t \in (0,T)$ and all $\rho>0$, 
\begin{align*}
D(x,t,\rho)&\ge \rho^{m-1},\\
D(x,t,\rho)&\le (e^{\lambda_1T}\rho+1)^{m-1} 
                \le (e^{\lambda_1T}+1)^{m-1}\max\{\rho, 1\}^{m-1},\\
|S_1(x,t,\rho)| &\le \chi(e^{\lambda_1T}+1)^{p-1}\max\{\rho, 1\}^{p-1}.
\end{align*}
Moreover, the initial condition in \eqref{P} implies that 
$\int_\Omega \widetilde{u}(\cdot, 0)=\int_\Omega u_0=M_0$. 
Here we choose $\theta>n$ satisfying
\begin{align*}
m-p \in \Big(\frac{1}{\theta}-1-\frac{1}{n},\ \frac{1}{\theta}-\frac{2}{n}\Big]
\end{align*}
and
\begin{align*}
\sigma&=\frac{n(n-1)}{(m-p+1)n+1}+\ep\\
&>\frac{n(n-1)}{(m-p+1)n+1-\frac{n}{\theta}}
=\frac{n-1}{(m-p)+1+\frac{1}{n}-\frac{1}{\theta}}.
\end{align*}
Since $p \ge q$ and 
\begin{align*}
\Big|\frac{S_2(x,t,\rho)}{S_1(x,t,\rho)}\Big|
&=\frac{\xi(e^{\lambda_1t}\rho+1)^{q-2}\rho}{\chi(e^{\lambda_1t}\rho+1)^{p-2}\rho}
=\frac{\xi}{\chi}(e^{\lambda_1t}\rho+1)^{q-p}
\le \frac{\xi}{\chi},
\end{align*}
for all $x \in \Omega$, $t \in (0,T)$ and all $\rho>0$, 
following the steps in the proof of \cite[Lemma~5.2]{BFL-2021}, 
we establish
\begin{align*}
\int_\Omega |x|^{(n-1)\theta}|\mathbf{f}(x,t)|^\theta\,dx
&\le c_1\Big(\frac{\alpha}{\beta}+\frac{\xi}{\chi}\cdot\frac{\gamma}{\delta}\Big)^\theta\Big(\frac{2e^{\lambda_1T}M_0}{\omega_{n-1}}\Big)^\theta |\Omega|
\end{align*}
for all $t \in (0,T)$ with some $c_1>0$, where $\omega_{n-1}$ denotes the $(n-2)$-dimensional surface area of the unit sphere in $\R^{n-1}$. 
Thanks to \cite[Theorem~1.1]{F-2020}, we derive that there exists $c_2>0$ such that
$\widetilde{u}(x,t) \le c_2|x|^{-\sigma}$ for all $x \in \Omega$ and all $t \in (0,T)$. 
This leads to the end of the proof.
\end{proof}

We now introduce the mass accumulation functions $U=U(s, t), V=V(s,t)$ and $
W=W(s,t)$ 
as follows:
\begin{align}
U(s,t)&:=\int_0^{s^{\frac1n}} \rho^{n-1}u(\rho, t)\,d\rho,\label{Udef}\\
V(s,t)&:=\int_0^{s^{\frac1n}} \rho^{n-1}v(\rho, t)\,d\rho\\
\intertext{and}
W(s,t)&:=\int_0^{s^{\frac1n}} \rho^{n-1}w(\rho, t)\,d\rho, \label{Wdef}
\end{align}
where $s:=r^n$ for $r \in [0, R]$ and $t \in [0, \tmax)$. 
We next define the moment-type functional 
\begin{align}\label{phi}
\phi(s_0, t):=\int_0^{s_0}s^{-b}(s_0-s)U(s,t)\,ds
\end{align}
for $s_0 \in (0, R^n)$, $t \in [0, \tmax)$ and $b \in (0,1)$. 

\begin{lem}\label{lem_DI_p>q}
Assume that $p>q$. 
Let $\mu_1>0$, $\kappa \ge 1$, 
$a \ge 0$ and $T>0$. 
Then there exist $C_1, C_2>0$ such that
\begin{align}\label{phiDI}
\frac{\pa \phi}{\pa t}(s_0, t) \ge &\ 
C_1\int_0^{s_0}s^{-b}(s_0-s)(nU_s(s,t)+1)^{p-2}U(s,t)U_s(s,t)\,ds \notag\\
&+n^2\int_0^{s_0}s^{2-\frac2n-b}(s_0-s)(nU_s(s,t)+1)^{m-1}U_{ss}(s,t)\,ds \notag\\
&-\chi\beta n\int_0^{s_0}s^{-b}(s_0-s)(nU_s(s,t)+1)^{p-2}V(s,t)U_s(s,t)\,ds \notag\\
&-n^{\kappa-1}\mu_1\int_0^{s_0}s^{-b}(s_0-s)\Big[\int_0^{s_0}\eta^{\frac an}U_s^\kappa(\eta,t)\,d\eta\Big]\,ds-C_2\phi(s_0,t)
\end{align}
for all $s_0 \in (0, R^n)$ and all $t \in (0, \min\{T, \tmax\})$. 
\end{lem}
\begin{proof}
The first equation in \eqref{prob} implies that  
$u=u(r,t)$, $v=v(r,t)$, $w=w(r,t)$ 
satisfy
\begin{align}\label{ueq}
u_t
&=\frac{1}{r^{n-1}}\big((u+1)^{m-1}r^{n-1}u_r\big)_r
    -\chi \frac{1}{r^{n-1}}\big(u(u+1)^{p-2}r^{n-1}v_r\big)_r\notag\\
&\quad\,
    +\xi \frac{1}{r^{n-1}}\big(u(u+1)^{q-2}r^{n-1}w_r\big)_r+\lambda u
  -\mu u^\kappa.
\end{align}
Moreover, the second and third equations in \eqref{prob} yield that 
\begin{align}
r^{n-1}v_r(r,t)&=\beta V(r^n, t)-\alpha U(r^n,t), \label{Veq}\\
r^{n-1}w_r(r,t)&=\delta W(r^n, t)-\gamma U(r^n,t) \label{Weq}
\end{align}
for all $r \in (0,R)$ and all $t \in (0,\tmax)$. 
Integrating \eqref{ueq} combined with \eqref{Veq} and \eqref{Weq} with respect to $r$ over $[0, s^{\frac1n}]$, 
we see from the nonnegativity of $\lambda$  and \eqref{mupro} that 
\begin{align}\label{UDI}
U_t \ge&\ n^2s^{2-\frac 2n}(nU_s+1)^{m-1}U_{ss}\notag\\
&+\chi nU_s(nU_s+1)^{p-2}(\alpha U-\beta V)-\xi nU_s(nU_s+1)^{q-2}(\gamma U-\delta W)\notag\\
&-n^{\kappa -1}\mu_1\int_0^s\eta^{\frac an}U_s^\kappa(\eta,t)\,d\eta
\end{align}
for all $s \in (0,R^n)$ and all $t \in (0,\tmax)$. 
Combining \eqref{phi} and \eqref{UDI}, we obtain 
\begin{align}\label{phiDI2}
\frac{\pa \phi}{\pa t}(s_0, t) &\ge
\chi\alpha n\int_0^{s_0}s^{-b}(s_0-s)(nU_s(s,t)+1)^{p-2}U(s,t)U_s(s,t)\,ds \notag\\
&\quad\,-\xi\gamma n\int_0^{s_0}s^{-b}(s_0-s)(nU_s(s,t)+1)^{q-2}U(s,t)U_s(s,t)\,ds \notag\\
&\quad\,+n^2\int_0^{s_0}s^{2-\frac2n-b}(s_0-s)(nU_s(s,t)+1)^{m-1}U_{ss}(s,t)\,ds \notag\\
&\quad\,-\chi\beta n\int_0^{s_0}s^{-b}(s_0-s)(nU_s(s,t)+1)^{p-2}V(s,t)U_s(s,t)\,ds \notag\\
&\quad\,+\xi\delta n\int_0^{s_0}s^{-b}(s_0-s)(nU_s(s,t)+1)^{q-2}W(s,t)U_s(s,t)\,ds \notag\\
&\quad\,-n^{\kappa-1}\mu_1\int_0^{s_0}s^{-b}(s_0-s)\Big[\int_0^{s_0}\eta^{\frac an}U_s^\kappa(\eta,t)\,d\eta\Big]\,ds\notag\\
&=: J_1-J_2+J_3-J_4+J_5-J_6
\end{align}
for all $s_0 \in (0, R^n)$ and all $t \in (0, \min\{T, \tmax\})$. 
Here we estimate the term $J_2$. 
We first consider the case $q>1$. 
In this case, 
using Young's inequality, we see that 
for all $\ep_1>0$ there exists $c_1(\ep_1)>0$ such that
\begin{align}\label{key1}
(nU_s(s,t)+1)^{q-2}U_s(s,t) 
&\le \ep_1\Big[(nU_s(s,t)+1)^{(q-1)-1}U_s(s,t) \Big]^{\frac{p-1}{q-1}}+c_1(\ep_1) \notag\\
&=\ep_1(nU_s(s,t)+1)^{p-1-\frac{p-1}{q-1}}U_s^{\frac{p-1}{q-1}}(s,t) +c_1(\ep_1).
\end{align}
Here we notice from the relation 
$\frac{p-1}{q-1}>1$ by $p>q>1$ that 
\begin{align}\label{key2}
U_s^{\frac{p-1}{q-1}}(s,t)
=U_s^{\frac{p-1}{q-1}-1}(s,t)U_s(s,t) \le (nU_s(s,t)+1)^{\frac{p-1}{q-1}-1}U_s(s,t).
\end{align}
A combination of \eqref{key1} and \eqref{key2} implies that 
\begin{align}\label{J2est3}
(nU_s(s,t)+1)^{q-2}U_s(s,t)&\le\ep_1(nU_s(s,t)+1)^{p-2}U_s(s,t)+c_1(\ep_1).
\end{align}
In the case $q\le 1$, noting that 
\begin{align*} 
(nU_s(s,t)+1)^{q-2} U_s(s,t)
\le (nU_s(s,t)+1)^{-1}U_s(s,t)
\le n^{-1},
\end{align*}
we can choose $\ep_1=0$ and $c_1(\ep_1)=n^{-1}$ in the estimate \eqref{J2est3}. 
In view of \eqref{J2est3} we obtain
\begin{align}\label{J2est4}
J_2 &=\xi\gamma n\int_0^{s_0}s^{-b}(s_0-s)(nU_s(s,t)+1)^{q-2}U(s,t)U_s(s,t)\,ds\notag\\
&\le \ep_1\xi\gamma n\int_0^{s_0}s^{-b}(s_0-s)(nU_s(s,t)+1)^{p-2}U(s,t)U_s(s,t)\,ds\notag\\
&\quad\,+c_1(\ep_1)\int_0^{s_0}s^{-b}(s_0-s)U(s,t)\,ds \notag\\
&=\ep_1\xi\gamma n\int_0^{s_0}s^{-b}(s_0-s)(nU_s(s,t)+1)^{p-2}U(s,t)U_s(s,t)\,ds+c_1(\ep_1)\phi(s_0,t).
\end{align}
Combining \eqref{J2est4} with \eqref{phiDI2} and noting that $J_5 \ge 0$, we establish
\begin{align*}
\frac{\pa \phi}{\pa t}(s_0, t) &\ge  
(\chi\alpha-\ep_1\xi\gamma) n\int_0^{s_0}s^{-b}(s_0-s)(nU_s(s,t)+1)^{p-2}U(s,t)U_s(s,t)\,ds \notag\\
&\quad\,+n^2\int_0^{s_0}s^{2-\frac2n-b}(s_0-s)(nU_s(s,t)+1)^{m-1}U_{ss}(s,t)\,ds \notag\\
&\quad\,-\chi\beta n\int_0^{s_0}s^{-b}(s_0-s)(nU_s(s,t)+1)^{p-2}V(s,t)U_s(s,t)\,ds \notag\\
&\quad\,-n^{\kappa-1}\mu_1\int_0^{s_0}s^{-b}(s_0-s)\Big[\int_0^{s_0}\eta^{\frac an}U_s^\kappa(\eta,t)\,d\eta\Big]\,ds\notag\\
&\quad\,-c_1(\ep_1)\phi(s_0,t)
\end{align*}
for all $s_0 \in (0, R^n)$ and all $t \in (0, \min\{T, \tmax\})$. 
Here, choosing $\ep_1:=\frac{\chi\alpha}{2\xi\gamma}$ when $q>1$ and recalling that 
$\ep_1=0$ when $q \le1$, we see that $\chi\alpha-\ep_1\xi\gamma>0$,
which means that the desired inequality 
\eqref{phiDI} holds. 
\end{proof}

\begin{prth4.1} 
Let $\sigma:=\frac{n(n-1)}{(m-p+1)n+1}+\ep$ with some $\ep>0$ (fixed later) 
and let $u_0(x) \le L|x|^{-\sigma}$ 
for all $x \in \Omega$. 
Then Lemma~\ref{lem_profile} implies that 
\eqref{profile2} holds: 
$u(x,t) \le C|x|^{-\sigma}$ for all $x \in \Omega$ and all $t \in (0,T)$ with some $C>0$. 
Also, to estimate the first four terms on 
the right-hand side of \eqref{phiDI} 
we follow the steps in \cite[Lemmas~3.3,~3.6,~3.7~and~3.9]{T-2021}. 
Employing those estimates in our case, 
we have that there exist $c_1, c_2, c_3>0$ such that
\begin{align*}
\frac{\pa \phi}{\pa t}(s_0, t) 
&\ge c_1\psi_p(s_0,t)\notag\\
&\quad\, -c_2s_0^{\frac{3-b}{2}-\frac2n-\frac{\sigma}{2n}[2(m-1)_++(2-p)_+]}\sqrt{\psi_p(s_0,t)}-c_2s_0^{3-\frac2n-b}\notag\\
&\quad\,-c_2s_0^{\frac2n+\frac{1-b}{2}-\frac{\sigma}{2n}[(2-p)_++2(p-2)_+]}\sqrt{\psi_p(s_0,t)}-c_2s_0^{\frac2n-\frac{\sigma}{n}[(2-p)_++(p-2)_+]}\psi_p(s_0,t)\notag\\
&\quad\, -c_2s_0^{\frac{3-b}{2}+\frac an-\frac{\sigma}{2n}[2(\kappa-1)+(2-p)_+]}\sqrt{\psi_p(s_0,t)}\notag\\
&\quad\,-c_3\phi(s_0,t)
\end{align*}
for all $s_0 \in (0, R^n)$ and all $t \in (0, \min\{T, \tmax\})$, 
where 
\begin{align*}
\psi_p(s_0,t):=\int_0^{s_0}s^{-b+\frac{\sigma}{n}(2-p)_+}(s_0-s)U(s,t)U_s(s,t)\,ds
\end{align*}
for $s_0 \in (0, R^n)$ and $t \in [0, \tmax)$. 
We take $\ep_1>0$ which will be fixed later. 
Using the Young inequality, we can see that
there exists $c_4(\ep_1)>0$ such that
\begin{align}\label{phiDI**}
\frac{\pa \phi}{\pa t}(s_0, t) 
&\ge c_1\psi_p(s_0,t)-\ep_1\psi_p(s_0, t)
-c_2s_0^{\frac2n-\frac{\sigma}{n}[(2-p)_++(p-2)_+]}\psi_p(s_0,t)\notag\\
&\quad\, -c_4(\ep_1)\Big(s_0^{3-b-\frac4n-\frac \sigma n[2(m-1)_++(2-p)_+]}+s_0^{2-\frac2n-b}\notag\\
&\qquad\qquad\quad\ +s_0^{\frac4n+1-b-\frac\sigma n[(2-p)_++2(p-2)_+]}+s_0^{3-b+\frac{2a}{n}-\frac\sigma n[2(\kappa-1)+(2-p)_+]}\Big)\notag\\
&\quad\,-c_3\phi(s_0,t)
\end{align}
for all $s_0 \in (0, R^n)$ and all $t \in (0, \min\{T, \tmax\})$. 
We now pick $s_1 \in (0, R^n)$ small enough such that 
\begin{align*}
c_2s_0^{\frac2n-\frac{\sigma}{n}[(2-p)_++(p-2)_+]}\psi_p(s_0,t)
\le \frac{1}{4}c_1\psi_p(s_0,t)
\end{align*}
for all $s_0 \in (0, s_1)$ and all $t \in (0, \min\{T, \tmax\})$, 
and set $\ep_1:=\frac{c_1}{4}$. 
Then we have from \eqref{phiDI**} that 
\begin{align*}
\frac{\pa \phi}{\pa t}(s_0, t) 
&\ge \frac{1}{2}c_1\psi_p(s_0,t)\notag\\
&\quad\, -c_4\Big(s_0^{3-b-\frac4n-\frac \sigma n[2(m-1)_++(2-p)_+]}+s_0^{2-\frac2n-b}\notag\\
&\qquad\quad\ \ +s_0^{\frac4n+1-b-\frac\sigma n[(2-p)_++2(p-2)_+]}+s_0^{3-b+\frac{2a}{n}-\frac\sigma n[2(\kappa-1)+(2-p)_+]}\Big)\notag\\
&\quad\,
-c_3\phi(s_0,t)
\end{align*}
for all $s_0 \in (0, s_1)$ and all $t \in (0, \min\{T, \tmax\})$. 
By an argument similar to that in the proof of \cite[Lemma~4.3]{T-2021}, 
thanks to the conditions \ref{C1}--\ref{C3}, 
we can pick $\ep_0>0$ and then 
for $\sigma=\frac{n(n-1)}{(m-p+1)n+1}+\ep_0$
there exist $c_5, c_6>0$ and $\theta \in (0,\ 2-\frac{\sigma}{n}(2-p)_+)$
such that
\begin{align}\label{phiDI***}
\frac{\pa \phi}{\pa t}(s_0, t) \ge \frac{1}{2}c_1\psi_p(s_0,t)
-c_5s_0^{3-b-\theta}-c_6\phi(s_0,t).
\end{align}
Applying the estimate
$\sqrt{\psi_p(s_0,t)} \ge c_7s_0^{\frac{b-3}{2}+\frac{\sigma}{2n}(2-p)_+}\phi(s_0,t)$
with some $c_7>0$ 
(see \cite[Lemma~3.10]{T-2021}) 
to the first term on the right-hand side of 
\eqref{phiDI***}, we have
\begin{align}\label{phiDIa}
\frac{\pa \phi}{\pa t}(s_0, t) \ge c_7s_0^{b-3+\frac \sigma n (2-p)_+}\phi^2(s_0,t)
-c_5s_0^{3-b-\theta}-c_6\phi(s_0,t)
\end{align}
for all $s_0 \in (0, s_1)$ and all $t \in (0, \min\{T, \tmax\})$. 
Again by Young's inequality, we derive that 
there exists $c_8>0$ such that
\begin{align}\label{phiDIb}
c_6\phi(s_0,t) 
\le \frac{1}{2}c_7s_0^{b-3+\frac \sigma n (2-p)_+}\phi^2(s_0,t)
+c_8s_0^{3-b-\frac \sigma n (2-p)_+}.
\end{align}
A combination of \eqref{phiDIa} and \eqref{phiDIb} 
yields 
\begin{align}\label{DIgoal}
\frac{\pa \phi}{\pa t}(s_0, t) &\ge \frac{1}{2}c_7s_0^{b-3+\frac \sigma n (2-p)_+}\phi^2(s_0,t)
-c_5s_0^{3-b-\theta}-c_8s_0^{3-b-\frac \sigma n (2-p)_+}\notag\\
&\ge \frac{1}{2}c_7s_0^{b-3+\frac \sigma n (2-p)_+}\phi^2(s_0,t)
-c_9s_0^{\widetilde{\theta}}
\end{align}
for all $s_0 \in (0, s_1)$ and all $t \in (0, \min\{T, \tmax\})$ 
with some $c_9>0$ and
$\widetilde{\theta}=\min\{3-b-\theta,\ 3-b-\frac \sigma n (2-p)_+\}$. 
Here, by the conditions \ref{C1}--\ref{C3},  we can take $b \in (0, 1)$ satisfying
\begin{align*}
b<2-\frac4n-\frac \sigma n[2(m-1)_++(2-p)_+]
\end{align*}
(see \cite[Lemma~4.1]{T-2021}). 
This yields that 
\begin{align*}
b-3+\frac \sigma n (2-p)_+
&<\Big\{2-\frac4n-\frac \sigma n[2(m-1)_++(2-p)_+]\Big\}-3+\frac \sigma n (2-p)_+\notag\\
&=-1-\frac 4n-\frac{2\sigma}{n}(m-1)_+<0\\
\intertext{and moreover, recalling the choice that 
$\theta \in (0,\ 2-\frac{\sigma}{n}(2-p)_+)$, we have}
3-b-\theta
&>3-b-\Big[2-\frac{\sigma}{n}(2-p)_+\Big]\notag\\
&=1-b+\frac{\sigma}{n}(2-p)_+>0,
\end{align*}
which lead to $\widetilde{\theta}>0$. 
Taking into account 
the proof of \cite[Theorem~1.1]{W-2018} or 
\cite[Theorem~1.1]{BFL-2021}, 
we obtain $\tmax<T<\infty$, 
which implies that  Theorem~\ref{thm_bu1} holds 
by virtue of the criterion \eqref{bc}.  \qed
\end{prth4.1}


\subsection{The case \boldmath{$p=q$}}
\label{Subsec4.2}

In this subsection we state the following theorem 
which guarantees finite-time blow-up
in \eqref{prob} in the case $p=q$.

\begin{thm}\label{thm_bu2}
Assume that $p=q$ and $\chi\alpha-\xi\gamma>0$. 
Moreover, suppose that $m$, $p$ and $\kappa$ fulfill 
the same conditions as in 
Theorem~{\rm \ref{thm_bu1}}. 
Let $M_0>0$, $M_1 \in (0, M_0)$ and $L>0$.  Then the conclusion of Theorem~{\rm \ref{thm_bu1}} holds.
\end{thm}

In order to prove the above theorem we show the following lemma giving the pointwise lower estimate for $\frac{\pa \phi}{\pa t}$, 
where $U, V, W$ and $\phi$ are defined 
as in \eqref{Udef}--\eqref{Wdef} and \eqref{phi}, 
respectively. 

\begin{lem}\label{lem_DI_p=q}
Suppose that $p=q$. 
Let $\mu_1>0$, $\kappa \ge 1$, 
$a \ge 0$ and $T>0$. 
Then there exist $C_1, C_2>0$ such that
\begin{align}\label{phiDI6}
\frac{\pa \phi}{\pa t}(s_0, t) \ge &\ 
(\chi\alpha-\xi\gamma) n\int_0^{s_0}s^{-b}(s_0-s)(nU_s(s,t)+1)^{p-2}U(s,t)U_s(s,t)\,ds \notag\\
&+n^2\int_0^{s_0}s^{2-\frac2n-b}(s_0-s)(nU_s(s,t)+1)^{m-1}U_{ss}(s,t)\,ds \notag\\
&-\chi\beta n\int_0^{s_0}s^{-b}(s_0-s)(nU_s(s,t)+1)^{p-2}V(s,t)U_s(s,t)\,ds \notag\\
&-n^{\kappa-1}\mu_1\int_0^{s_0}s^{-b}(s_0-s)\Big[\int_0^{s_0}\eta^{\frac an}U_s^\kappa(\eta,t)\,d\eta\Big]\,ds
\end{align}
for all $s_0 \in (0, R^n)$ and all $t \in (0, \min\{T, \tmax\})$. 
\end{lem}

\begin{proof}
Arguing as in Lemma~\ref{lem_DI_p>q}, we have \eqref{phiDI2}
with $q=p$. 
We then rearrange it as 
\begin{align*}
\frac{\pa \phi}{\pa t}(s_0, t) \ge &\ 
(\chi\alpha-\xi\gamma) n\int_0^{s_0}s^{-b}(s_0-s)(nU_s(s,t)+1)^{p-2}U(s,t)U_s(s,t)\,ds \notag\\
&+n^2\int_0^{s_0}s^{2-\frac2n-b}(s_0-s)(nU_s(s,t)+1)^{m-1}U_{ss}(s,t)\,ds \notag\\
&-\chi\beta n\int_0^{s_0}s^{-b}(s_0-s)(nU_s(s,t)+1)^{p-2}V(s,t)U_s(s,t)\,ds \notag\\
&-n^{\kappa-1}\mu_1\int_0^{s_0}s^{-b}(s_0-s)\Big[\int_0^{s_0}\eta^{\frac an}U_s^\kappa(\eta,t)\,d\eta\Big]\,ds
\end{align*}
for all $s_0 \in (0, R^n)$ and all $t \in (0, \min\{T, \tmax\})$, which means that \eqref{phiDI6} holds. 
\end{proof}

\begin{prth4.2}
In view of Lemma~\ref{lem_DI_p=q}, proceeding similarly in 
the proof of Theorem~\ref{thm_bu1} and 
taking $\sigma$ properly, we can find $c_1, c_2>0$ 
and $\theta \in (0,\ 2-\frac\sigma n(2-p)_+)$ 
such that 
\begin{align*}
\frac{\pa \phi}{\pa t}(s_0, t) \ge c_1s_0^{b-3+\frac \sigma n (2-p)_+}\phi^2(s_0,t)
-c_2s_0^{3-b-\theta}
\end{align*}
for all $s_0 \in (0, s_1)$ and all $t \in (0, \min\{T, \tmax\})$ for some small $s_1>0$.  
This inequality corresponds to \eqref{DIgoal} 
and proves Theorem~\ref{thm_bu2}. 
 \qed
\end{prth4.2}


\end{document}